\documentclass[final,3p]{elsarticle}

\usepackage[]{natbib}
\biboptions{sort&compress}
\usepackage{amsmath}
\usepackage{amssymb}
\usepackage{url}
\usepackage{amsthm}
\usepackage{mathtools}
\usepackage{xcolor}
\usepackage{soul} 
\usepackage{caption} 

\usepackage{graphicx}
\usepackage{subcaption}
\usepackage{xcolor,colortbl}
\usepackage{pifont} 

\usepackage[colorinlistoftodos,prependcaption,textsize=tiny]{todonotes}
\usepackage[linesnumbered, ruled]{algorithm2e}

\newtheorem{thm}{Theorem}

\newtheorem{obs}[thm]{Observation}
\newtheorem{lemma}[thm]{Lemma}
\newtheorem{conj}[thm]{Conjecture}
\newtheorem{prop}[thm]{Proposition}
\newtheorem{cor}[thm]{Corollary}

\theoremstyle{remark}
\newtheorem*{rem}{Remark}%
\newcommand\prob[3]{%
\vspace{\parindent}
  \hspace{\parindent} {\bf #1} \par
  \hspace{\parindent} {\bfseries Input}: #2\par
  \hspace{\parindent} {\bfseries Output}: #3\par
  \vspace{\parindent}
}

\newcommand{\gf}[1]{\mathbf{#1}}                
\newcommand{\gfroot}[2]{(#1, #2)}     
\newcommand{\gfo}[1]{\mathbf{G}^{#1}}           
\newcommand{\graphcount}[3]{deg_{#1}(#2, #3)}   

\newcommand{\Hasym}[0]{H_{\text{asym}}}
\newcommand{\ecc}[1]{\epsilon(#1)}
\newcommand{\Graph}[1]{\mathcal{G}_{#1}}
\newcommand{\Graphlet}[1]{\mathcal{G'}_{#1}}
\newcommand{\gmvh}[1]{H \setminus \{ #1\}} 
\newcommand{\diam}[1]{\text{diam}(#1)}
\newcommand{\dist}[0]{\text{dist}}

\usepackage{float}

\begin{document}

\begin{frontmatter}
\title{Reconstructing graphs and their connectivity using graphlets}

\author[cas,mff]{David Hartman}
\ead{hartman@cs.cas.cz}

\author[cas,mff]{Aneta Pokorná}
\ead{pokorna@cs.cas.cz}

\author[cas,mff]{Daniel Trlifaj}
\ead{daniel.trlifaj@gmail.com}

\author[barc]{Lluís Vena}
\ead{lluis.vena@upc.edu}

\address[cas]{Institute of Computer Science of the Czech Academy of Sciences, \\ Pod Vod\'{a}renskou v\v{e}\v{z}\'{i} 271/2, 182 07 Prague, Czech Republic}
\address[mff]{Computer Science Institute of Charles University, Faculty of Mathematics and Physics, Charles University, Malostransk\'{e} n\'{a}m. 25, Prague 1, 118 00, Czech Republic}
\address[barc]{Serra Húnter Fellow, IMTech Fellow, Centre de Recerca Matemàtica Fellow, Department of Mathematics,
Universitat Politècnica de Catalunya - BarcelonaTech (UPC)}

\vspace{-1em}

  \date{\today}

  \begin{abstract} 
  Graphlets are subgraphs rooted at a fixed vertex. The number of occurrences of graphlets aligned to a particular vertex, called graphlet degree sequence (gds), gives a topological description of the surrounding of the analyzed vertex. Graphlet degree distribution (gdd) of a graph is a matrix containing graphlet degree sequence for all vertices in the given graph.
  A long standing open problem called reconstruction conjecture (RC) asks whether the structure of a graph is uniquely determined by the multiset of its vertex-deleted subgraphs. Graphlet degree distribution up to size $(n-1)$, $(\leq n-1)$-gdd, gives more information to reconstruct the graph and we use it to reconstruct any graph having a unique almost-asymmetric vertex-deleted subgraph, where almost-asymmetric means that at most one automorphism orbit has size larger than one. 
  Moreover, we prove that any graph containing a vertex-cut of size $1$ or any graph of order $n$ having a vertex with degree at most $2$ or at least $n-2$ is reconstructible from its $(\leq n-1)$-gdd, which expands results shown in the standard RC.
  We also discuss the relation between gdd and graph connectivity and the conditions on $(\leq 3)$-gdd, whose breaking means that no graph with such gdd exists.
\end{abstract}


\end{frontmatter}


\section{Introduction}
The analysis of the local topology of graphs using the frequencies of small rooted subgraphs called graphlets has been introduced by Przulj \cite{przulj2004}. This tool has been successfully applied 
to some real-world problems such as to improve the random model of protein-protein interaction networks \cite{protein-align-Przulj},
to understand the topology of social networks \cite{janssen2012}, to predict connections of microRNAs with diseases \cite{gimda}, or to analyze the structure of the human brain \cite{Finotelli2021}.
These applications suggest that graphlets are a useful tool to compare the local topological similarity of vertices.

The graphlet analysis lies in counting $(\leq k)$-graphlet degree sequence of every vertex $v$, that is, the counts of induced occurrences of graphlets of sizes at most $k$ where the root is matched to $v$. The matrix containing $(\leq k)$-sequence of every vertex in the given graph is called its $(\leq k)${\it -graphlet degree distribution (gdd)}.

The question of whether the structure of a graph can be recovered from the multiset of its vertex-deleted subgraphs, called the {\it deck}, known as the reconstruction conjecture (RC), is one of the long-standing unresolved conjectures in graph theory. 
In his PhD thesis, P.J. Kelly proved reconstruction conjecture for regular graphs, Eulerian graphs and trees. \cite{kelly1942isometric} 
Later, the conjecture has been proved for outer planar graphs\cite{Giles1974outerplanar},
 separable graphs without end-vertices\cite{Bondy1969OnUC},
or maximal planar graphs\cite{Lauri1981max-planar}.
It has been even shown that all graphs are reconstructible if and only if all $2$-connected graphs with radius at most two are reconstructible.\cite{rec-comp-rad-2}

We examine an alternative version of RC, asking whether a graph is reconstructible given its $(\leq n-1)$-{\it graphlet degree distribution} (gdd). The intuition behind this effort is that the information about the number of rooted subgraphs for respective vertices seems to provide richer information compared to the deck. We can support this intuition by showing that for the $2$-connected case, the deck can be easily obtained from the $(\leq n-1)$-gdd. 

Originally, it was assumed that solving RC for asymmetric graphs would be easier than for symmetric ones (\cite{recnumsur}, page 446). However, there are no results related to the standard reconstruction conjecture using asymmetry. 
We show that if a $2$-connected graph $H$ contains a vertex-deleted asymmetric subgraph $\Hasym$ that can be obtained by either removal of a unique vertex or if all the vertices whose removal leads to $\Hasym$ share the same neighbourhood, then $H$ can be reconstructed from its $(n-1)$-gdd.
Hence, we prove graphlet reconstructibility for a family of graphs that has not yet been shown to be reconstructible from the deck.
As almost all finite graphs are asymmetric \cite{erdosrenyi1963asymmetric}, there are also many vertex-deleted asymmetric subgraphs, so our result allows reconstruction of potentially many graphs.

In this article, section \ref{s:prop} describes some basic properties of graphlet degree sequences. Then in section \ref{s:mot} we explain the relation of graphlets and motifs, other commonly used tool in network analysis.
The links between graphlet degree distribution and connectivity are described in section \ref{s:conn}. 
Section \ref{s:rec} shows the reconstructions of graphs having vertex-cut of size one, graphs having a vertex of degree at most 2 or at least $(n-2)$ and graphs containing specific asymmetric or almost-asymmetric vertex-deleted subgraphs.
Finally, sections \ref{s:uniq} and \ref{s:dec} are about the uniqueness of graphlet degree sequences and examine the decision problem whether there exists a graph with given graphlet degree sequence.

\section{Preliminaries}
Throughout this work, $H = (V, E)$ denotes an unoriented graph with vertices $V$ and edges $E \subseteq { V \choose 2 }$. 
The situation in which the vertices $u$ and $v$ are connected by an edge will be denoted as $\{u,v\} \in E$ or, more simply, $uv\in E$. 
Unless otherwise stated, $n = |V|$. For $v \in V$, $N_H(v)$ is the {\it neighbourhood} of $v$, i.e. $\{w\, |\, vw \in E\}$. For this and the following definitions, the subscript $H$ is omitted whenever $H$ is clear from the context. 
Then $\deg_H(v) = |N_H(v)|$ is the {\it degree} of a vertex $v$. The distance $d_H(u,v)$ between two vertices $u,v \in V(H)$ is the length of the shortest path between $u$ and $v$. The {\it eccentricity} $\ecc{v}$ is $\max_{u \in V(H)} \{d(u,v)\}$ and $\diam{G} = \max_{v \in V(H)} \ecc{v}$. The symbol $\cong$ stands for graph isomorphism. The set of all non-isomorphic graphs is denoted by $\Graph{}$, $\Graph{n} = \{G \in \Graph{}\, |\ |V(G)| = n\}$ and $\Graph{\leq n} = \{G \in \Graph{}\, |\ |V(G)| \leq n\}$.

A {\it graphlet} rooted in $r$ is a pair $(G, r)$ where $G$ is a connected\footnote{There exists a notion of {\it graphettes}, basically graphlets that are not necessarily connected.\cite{graphettes}} subgraph of $H$ with $|V(G)| < n = |V(H)|$ and $r \in V(G)$. We call the graph $G$ the {\it underlying graph} of the graphlet $(G,r)$ and assume that there exists a function $U((G,r)) = G$. The {\it size} of a graphlet is understood as the size of the underlying graph, $|(G,r)| = |V(G)|$. Let $\Graphlet{}$ denote the set of all non-isomorphic graphlets, $\Graphlet{n} = \{G \in \Graphlet{}\, |\ |G| = n\}$, $\Graphlet{\leq n} = \{G \in \Graphlet{}\, |\ |G| \leq n\}$ and $\Graphlet{[H]}$ is the set of all graphlets with underlying graph $H$. Graphlets $(G, r)$ and $(H, s)$ are {\it isomorphic}, denoted as $(G,r) \cong (H,s)$, if $G \cong H$ and the respective isomorphism maps $r$ to $s$. To refer about an isomorphism class of graphs, we use a fixed ordering $\gamma$ of $\Graph{}$, such that $\gamma(G) < \gamma(H)$ whenever $|V(G)| < |V(H)|$ and $\gamma(G) = \gamma(H)$ iff $G \cong H$. The graph $H$ with $\gamma(G) = i$ is denoted by $G^i$.
Similarly, to identify an isomorphism class of graphlets, we use a fixed ordering $\vartheta$ of $\Graphlet{}$, such that $\vartheta((G,r)) < \vartheta((H,s))$ whenever $\gamma(G) < \gamma(H)$ and $\vartheta((G,r)) = \vartheta((H,s))$ iff $(G,r) \cong (H,s)$. The graphlet $(G,r)$ with $\vartheta((G,r)) = i$ is denoted by $\gfo{i}$ and its underlying graph is $G = U(\gfo{i})$. 
There are several choices of the ordering $\vartheta$. For graphlets up to size $5$, it was given explicitly by \cite{przulj2004}, see Figure \ref{f:orderings}. 
However, \cite{orbit-counting} gave a similar ordering, which is deterministically defined for graphlets of any size.

\begin{figure}
    \centering
    \includegraphics[width=0.8\linewidth]{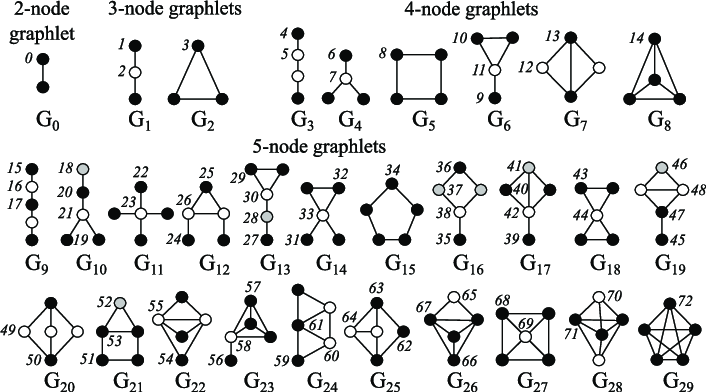}
    \caption{The orderings $\gamma$ of graphs $\Graph{}$ and $\vartheta$ of graphlets $\Graphlet{}$ up to size $5$ given by Pržulj \cite{przulj2007}.
Under each graph $F$ there is a label $G_{\gamma(F)}$.
Automorphism orbits are distinguished by different shades of black, white and gray. 
The number $\vartheta((F, r))$ assigned to a graphlet \gfroot{F}{r} is written next to some vertex from the automorphism orbit of $r$.
Image source : \cite{L-GRAAL}.
}
    \label{f:orderings}
\end{figure}

We say that a vertex $v \in V(H)$ {\it touches } a graphlet $(G,r)$, if there exists an embedding $e$ of $G$ inside $H$ such that $e(r) = v$. The {\it graphlet degree} $\graphcount{H}{\gfroot{G}{r}}{v}$ of a vertex $v \in V(H)$ and graphlet $\gfroot{G}{r}$ in $H$ is the number of times $v$ touches $(G,r)$. The $\leq k$-{\it graphlet degree sequence} (shortly $\leq k$-gds) of $v \in V(H)$ is a vector of values $(\graphcount{H}{\gfo{i}}{v}\ |\ i \in \{\vartheta(\Graphlet{\leq k})\})$. Finally, $\leq k$-{\it graphlet degree distribution} (or shortly $\leq k$-gdd) of a graph $H$ is a matrix $D$ with dimensions ${|V(H)|\times |\vartheta(\Graphlet{\leq k})|}$ where $D_{i,j} = \graphcount{H}{\gfo{j}}{i}$.

\section{Properties of graphlet counts and graphlet degree sequences} \label{s:prop}

The $(\leq n-1)$-graphlet degree distribution can be viewed as a mapping $\text{gdd: } G \rightarrow D^{n\times k}$ where $G \in \Graph{}, |V(G)| = n$ and $k = \max_{\gfo{} \in \Graphlet{n-1}} \vartheta(\gfo{})$. Then it is natural to ask for an upper bound on the elements of $D$. 

\begin{obs}
    Let $D^{n\times k}$ be $(\leq n-1)$-gdd of $G$. Then for all $i \in [n], j \in [k]$ holds
    $$D_{i,j} \leq {n-1 \choose |\gfo{j}| - 1}$$
    and the maximum is attained for $G \cong K_n$.
\end{obs}
\begin{proof}
 The expression ${n-1 \choose |\gfo{j}| - 1}$ gives the upper bound on the number of graphlets of size $|\gfo{j}|$ because every choice of $|\gfo{j}|-1$ vertices of $V(G) \setminus \{i\}$ ($i$ is included automatically as the root) may or may not result in a valid graphlet (due to the resulting graph not being connected).
 
 In the complete graph, every possible graphlet size is present and every such graphlet is of the same form of a rooted complete subgraph of the given size, which has only one automorphism orbit.
\end{proof}

\begin{cor}
Let $D^{n\times k}$ be $(\leq n-1)$-gdd of $G$. Then for all $i \in [n]$,
     $$\sum D_{i,*} = \sum_{j \leq k} D_{i,j} \leq \sum_{p \leq n-1} {n-1 \choose p - 1}$$
\end{cor}

The complete graph serves as a showcase of a graph with maximal attainable values in gdd. On the other hand, what are the graphs with the lowest possible upper bound on the values of gdd?

\begin{obs}
Let $D^G$ be the $(n-1)$-gdd of a connected graph $G$ on $n \geq 3$ vertices. Then
$$\min_{G \in \Graph{n}} \max_{i \in [n], j \in [k]} D^G_{i,j} = 2$$
where $k = \max_{\gfo{} \in \Graphlet{n-1}} \vartheta(\gfo{})$ is the maximal index of a graphlet smaller than $G$.
The value is obtained for $G \cong P_n$ and $G \cong C_n$.
\end{obs}
\begin{proof}
    A lower value cannot be obtained because in connected graphs on $n \geq 3$ vertices, there is always a vertex of degree two, which means $D_{v,0} = 2$ for $v$ with $\deg_G(v)=2$.

    Now we show that $\max_{i \in [n], j \in [k]} D^G_{i,j} = 2$ holds for $G \cong P_n$ and $G \cong C_n$.
    The vertices in $G$ touch only graphlets whose underlying graph is a path of length $\leq n - 1$. 
    Suppose that there exists $v \in V(G)$ such that $v$ touches a graphlet weakly isomorphic to a path more than two times.
    Then the first two occurrences can have the longer part of the path to the left and right from $v$, but the existence of a third occurrence would require a branching at some point, implying the existence of a vertex $w$ with $\deg_G w \geq 3$, which is a contradiction to $G$ being a path or a cycle.
\end{proof}

\begin{obs}
    When comparing gds of a vertex $v$ in $H$ and $H \setminus x$ for some $x \in V(H)$, the frequency of any graphlet $\gfo{i}$ in the graphlet degree sequence of $v$ can only decrease. 
    Moreover, if $d = d(v,x)$ is the distance between $v$ and $x$, only the  frequency of graphlets with root eccentricity $d$ decreases.
\end{obs}
Especially, if the degree of root is the same as in $H$, most of the small graphlets ($\leq 5$ vertices) do not change at all.

\begin{obs}
    In a graphlet degree distribution of any connected graph $G$ of order $n$, there always exists a graphlet of size $(n-1)$.
\end{obs}
\begin{proof}
    For $x \in V(G)$, $G \setminus \{x\}$ is not an underlying graph of $(n-1)$-sized graphlets only if it is not connected. It cannot be the case that $G \setminus \{x\} \ \forall x \in V(G)$.
    Let $T$ be a spanning tree of $G$. Then for a pendant vertex $p$ of $T$, $G \setminus \{p\}$ is connected because $T \setminus \{p\}$ is connected and $T \subseteq G, |T| = |G| = n$.
\end{proof}

\section{Graphlets vs motifs}\label{s:mot}
A {\it motif} in a network is an induced subgraph with significantly higher number of appearances in the given network than is expected in a random network model. \cite{AlonMotifs}
Similarly to graphlet degree distribution, a network can be characterized by its {\it motif distribution}, a vector containing the number of induced occurrences of all types of (unrooted) subgraphs up to a given size $k$ in the whole graph.
Note that the same result can be achieved by adding the lines of $(\leq k)$-gdd, further summing all occurrences of weakly isomorphic graphlets and dividing each number by the size of the subgraph that it represents. The division is necessary due to the multiple counting of the same subgraph under different roots. 
Note that real motifs can be recognized by comparing the overall frequency of the subgraph with the expected frequency in a random graph model.

Is graphlet degree distribution matrix a more powerful characterization of the network than its motif distribution vector? 
The answer is positive, as is shown in Figure \ref{f:gdd-stronger-motifs}.

\begin{figure}
\begin{minipage}{0.5\textwidth}
\centering
\includegraphics{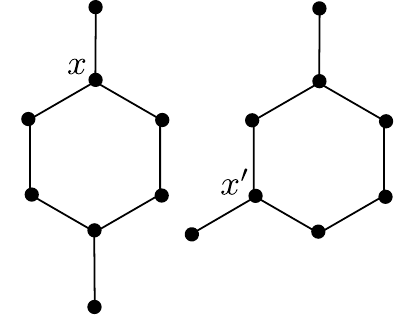}
\caption{Two graphs that cannot be distinguished by $\leq 4$ motif distribution, but can be distinguished by $\leq 4$-gdd, see Tables \ref{t:motifs} and \ref{t:graphlets}.}
\label{f:gdd-stronger-motifs}
\end{minipage}%
\begin{minipage}{0.5\textwidth}
    \centering
    \begin{tabular}{|c|c|c|c|c|c|c|c|c|c|}
    \hline
         $\gamma  $ & 0 & 1 & 2 & 3 & 4 & 5 & 6 & 7 & 8 \\
         $G_1$     & 8 & 10 & 0 & 10 & 2 & 0 & 0 & 0 & 0 \\
         $G_2$     & 8 & 10 & 0 & 10 & 2 & 0 & 0 & 0 & 0 \\
         \hline
    \end{tabular}
    \captionof{table}{The motif degree vector for $G_1$ and $G_2$ for graphs up to size $4$.}
    \label{t:motifs}
    \vspace{0.2cm}
    \begin{tabular}{|c|c|c|c|}
    \hline
    $\vartheta $ & 0 & 1 & 2  \\
    \hline
    $G_1$ & & & \\
    1=6 & 3 & 2 & 3  \\
    2=3=4=5 & 2 & 3 & 1 \\ 
    7=8 & 1 & 2 & 0  \\
    \hline
        $G_2$     &  &  &  \\
        1=4 & 3 & 2 & 3\\
        3=6 & 2 & 3 &1 \\
        7=8& 1& 2& 0\\
        2& 2& 4& 1\\
        5& 2& 2& 1\\
         \hline
    \end{tabular}
    \captionof{table}{The $(\leq 3)$-graphlet degree distribution sequences for $G_1$ and $G_2$.}
    \label{t:graphlets}
\end{minipage}

\end{figure}

\section{Relationship between gdd and graph connectivity}\label{s:conn}

The following lemma shows that $H$ must be $2$-vertex connected if we do not want to lose any information when moving from $H$ to its $(n-1)$-graphlet degree distribution. Otherwise, some of the vertex-deleted graphs become disconnected, split into larger number of smaller graphlets and are not present in the $(n-1)$-graphlet degree distribution of $H$.
\begin{lemma}\label{l:2-conn-whole-gds-from-n-1}
    Consider a connected graph, $H$, on $n$ vertices. If $H$ is $2$-vertex connected, $(n-1)$-gdd of $H$ determines the $(\leq n-2)$-gdd of $H$. 
\end{lemma}

\begin{rem}
This result has an analog in the RC, known as Kelly's lemma. It states that the number of induced subgraphs of a graph $G$ can be easily obtained from the deck of $G$.\cite{kelly1942isometric} \end{rem}

\begin{proof}
If $H$ is $2$-vertex connected, then any graphlet $\gfo{i}$ of size $\ell < (n-1)$ is induced exactly $(n-\ell)$-times in the $(n-1)$-graphlet degree distribution of $H$, each time obtained by removal of a different vertex from $V(H) \setminus V(U(\gfo{i}))$.
For each vertex $v$ and graphlet $\gfo{i}$ of size less than $(n-1)$, we obtain the entry $D_{v, i}$ of the $(\leq n-2)$-graphlet degree sequence by simply counting how many times $v$ touches $\gfo{i}$ in $(n-1)$-sized graphlets of $H$ rooted at $v$  and divide this number by $(n-|\gfo{i}|)$ to obtain $D_{v,i}$.
\end{proof}
Note it has been shown that if the reconstruction conjecture holds for all $2$-connected graphs, then it holds for all graphs. \cite{yang1988rec2conn} Lemma \ref{l:2-conn-whole-gds-from-n-1} can be easily generalized. 
\begin{lemma}\label{l:k-conn-det-n-k+1}
    Consider a connected graph, $H$, on $n$ vertices. If $H$ is $k$-vertex connected, $(n-k+1)$-graphlet degree sequence of $H$ determines $(\leq n-k)$-graphlet degree sequence of $H$. 
\end{lemma}

\begin{proof}
    When we fix a graphlet $\gfo{i} = \gfroot{G}{v}$ induced over the set $S \subseteq V(G), |S| = \ell$, it is induced in $n-\ell \choose k-1$ graphlets of size $(n-k+1)$ rooted in $v$, because we always include $S$ and choose $(k-1)$ vertices to exclude. By $H$ being $k$-vertex connected, each such exclusion results in a valid graphlet of size $(n-k+1)$ rooted in $v$ as $v \in S$. Hence we can proceed by summing the occurrences of $\gfroot{G}{v}$ in $(n-k+1)$-gds of $H$ to $D_{v,i}$, where $D$ is a matrix of dimensions corresponding to $(\leq n-k)$-gdd of $H$. The resulting graphlet degree distribution is obtained by dividing entry $D_{v,i}$ by $n-\ell \choose k-1$, where $\ell$ is the size of $\gfo{i}$.
\end{proof}

It turns out that $k$-vertex connectivity of a graph $H$ is closely related to the number of $(n-k+1)$-sized graphlets in $(\leq n-1)$-gdd of $H$.
\begin{prop} \label{p:k-conn-graphlet-num}
    Let $H$ be a graph, $|V(H)| = n$. Then $H$ is $k$-vertex connected if and only if 
    $$\sum_{\gfroot{G}{j} \in \Graphlet{n-k+1} } \graphcount{H}{\gfroot{G}{j}}{v} = {{n - 1} \choose {k - 1}}$$ for every $v \in V(H)$.
\end{prop}

\begin{proof}
    
Let $S \subseteq V(G)$, $|S| = k-1$. The fact that $H$ is $k$-vertex connected implies that $H \setminus \{S\}$ is also connected and forms a graphlet of size $(n-k+1)$. As there are ${n-1} \choose {k - 1}$ choices of $S \subseteq V(H)$ such that $v \not\in S$, there are ${n-1} \choose {k - 1}$ graphlets rooted at $v$ of size $(n-k+1)$.

The other way around, if there are ${n-1} \choose {k - 1}$ graphlets rooted at $v$ of size $(n-k+1)$ for any $v \in V(G)$, then $H \setminus \{S\}$ must be connected for any $S \subseteq V(H)$, $|S| = k - 1$ as in order to obtain the desired count, it must form a graphlet of size $(n-k+1)$.
\end{proof}

\begin{cor}\label{c:articulations}
    Let $H$ be a connected graph on $n$ vertices with its $(n-1)$-graphlet degree sequence $D^{n \times m}$. 
    The property
    \begin{equation}\label{e:(n-1)x}
        \sum_{i = 1}^{m} D_{v, i} = n-1
    \end{equation}
    holds for zero, one, or all vertices of $H$.
    These cases correspond to $H$ having more different articulations, precisely one articulation and no articulation, respectively.
    Moreover, if there is precisely one articulation, then it is the vertex having property \eqref{e:(n-1)x}.
\end{cor}

\begin{proof}
    Suppose $H$ is $2$-vertex connected. Then by Proposition \ref{p:k-conn-graphlet-num}, the property (\ref{e:(n-1)x}) holds for all vertices and there is no articulation.

    If the property (\ref{e:(n-1)x}) holds for at least one vertex $v \in V(H)$, but does not hold for a vertex $u \in V(H)$, then $H \setminus \{w\}$ is connected for any $w \neq v$, but $H \setminus \{v\}$ is disconnected. This implies that $v$ is an articulation of $H$ and that $v$ is the only vertex in $H$ for which property (\ref{e:(n-1)x}) holds because $H \setminus \{v\}$ cannot be a graphlet of size $(n-1)$ rooted in any $w \in V(H), w \neq v$ due to not being connected.
    This rules out the possibility that property (\ref{e:(n-1)x}) holds for $k$ vertices of $H$, where $1 < k < n$.

    The last possibility is that property (\ref{e:(n-1)x}) does not hold for any vertex, which is equal to the situation when there are $k$ articulations for $k > 1$. 
    If $G$ is rooted at the articulations, there are $n-1$ remaining vertices that can be removed, out of which $k-1$ are cut-vertices, whose removal leads to the creation of a graphlets strictly smaller than $(n-1)$. It follows that
    the largest possible number of $(n-1)$-sized graphlets is $(n-1 -(k-1)) = (n-k)$ and is obtained by rooting at the articulations. Other vertices have $(n-k-1)$ graphlets of size $(n-1)$. 
\end{proof}

\begin{cor}
    Let $H$ be a $(k-1)$-vertex connected graph on $n$ vertices with its $(n-k+1)$-graphlet degree sequence $D^{n \times m}$. 
    Let $p$ be the number of vertices that satisfy
    \begin{equation}\label{l:prop-p}
        \sum_{i = 1}^{m} D_{v, i} = {n-1 \choose k-1}.
    \end{equation}
    Then $p \in \{0, 1, \dots, k-1, n\}$.
     Furthermore, $p \leq k-1$ occurs when there exist $p$ vertices lying in every vertex cut of size $k$ and these vertices are identified by satisfying \eqref{l:prop-p}. 
    Finally, $p=n$ holds if $H$ is $k$-connected.
\end{cor}

\begin{proof}
    
    To obtain a graphlet of size $n-k+1$, $k-1$ vertices must be removed from $H$.
    Suppose that the given graphlet is rooted in $v$. Then those $k-1$ vertices are being removed from $V(H)\setminus {v}$, that is from the set of $n-1$ vertices.
    
    If there is no vertex cut of size $k-1$, then the removal of any set $S\subseteq V(G), |S| = k-1$ results in a connected graphlet and property (\ref{l:prop-p}) holds for every vertex as claimed in Proposition \ref{p:k-conn-graphlet-num}.
    
    Note that if there is a vertex cut $S$ of size $k-1$ in $H$, then the property (\ref{l:prop-p}) is invalidated for any $v \in V(G) \setminus S$, because $H \setminus S$ is not a graphlet of size $n-k+1$ rooted in $v$, as it is not connected.

    If there exist two vertex cuts $R,S$ of size $(k-1)$ in $H$ such that $R\, \cap\, S = \emptyset$, then property (\ref{l:prop-p}) is invalidated for every $v \in H$, because $\forall v \in V(H): v \notin R$ or $v \notin S$. 

    The only remaining case is that there are no two disjoint vertex cuts of size $(k-1)$. This implies that there are $j \leq k-1$ vertices $v_1, \dots, v_j$ such that $v_i \in S$ for every $(k-1)$-sized vertex cut $S$ of $H$, $\forall i \in [j]$. Due to being in every such cut and by $H$ having no vertex cut of size $(k-2)$, the property (\ref{l:prop-p}) is preserved for $v_1, \dots, v_j$. For $v \in V(G) \setminus \{v_1, \dots, v_j\}$, there always exists a $(k-1)$-sized vertex cut $S$ such that $v \in S$ (otherwise $v = v_i$ for some $i \in [j]$), which destroys property (\ref{l:prop-p}) for $v$.
\end{proof}

\section{Reconstructing graphs from $(\leq n-1)-gdd$}\label{s:rec}
Assuming that we know graphlets up to the size of $(n-1)$, can we reconstruct the underlying graph? This question is closely related to the reconstruction conjecture, a classical open problem in graph theory. 

We restate the conjecture as proposed by Kelly (1957) \cite{kelly1957congruence}: 

\begin{conj}
    Let $G$ be a graph on $n$ vertices. Let us denote $G_v$ a subgraph of $G$ formed by deleting vertex $v$ from $G$. Let $D(G) = \{G_v\,|\ v \in V(G)\}$ be a multiset of all vertex-deleted subgraphs of $G$, called deck. Then any two graphs $G$ and $H$ on $n>2$ vertices with equal multisets $D(G)$ and $D(H)$ are isomorphic. 
\end{conj}
The requirement $n>2$ is necessary as both graphs on $2$ verties, an edge and a pair of non-adjacent vertices have the same deck consisting of a pair of two $K_1$'s. However, this condition is not necessary if we restrict ourselves to connected graphs.
In the following subsections, we discuss the properties and classes of graphs reconstructible from the $(\leq n-1)$-gdd.

\subsection{Graph properties reconstructible from $(\leq n-1)$-gdd}
There are several properties of a graph obtainable from the $(\leq n-1)$-gdd of the graph. In this section, we provide several basic results that allow us to prove reconstruction results in the latter parts of this article.

\begin{obs}\label{o:edges}
    Let $G$ be a graph and $D$ be the $(\leq n-1)$-gdd of $G$.
    Then the number of edges $e=|E(G)|$ of $G$ is reconstructible from $D$.
\end{obs}
\begin{proof}
    We utilize that $\deg_G(v) = D_{v,0}$ and that $\sum_{v \in V(G)} \deg_G(v) = 2|E(G)|$.
\end{proof}

Reconstructing the number of edges of $G$ allows us to reconstruct the degree of a missing vertex $x$ in any graphlet of size $(n-1)$.
\begin{cor}\label{c:degree}
Let $G$ be a graph, $D$ be the $(\leq n-1)$-gdd of $G$ and $\gf{G}$ be an $(n-1)$-sized graphlet with $U(\gf{G}) \cong G \setminus \{x\}$ for a fixed $x \in V(G)$. Then $\deg_G(x)$ is reconstructible from $D$.
\end{cor}
\begin{proof}
    The degree of $x$ is obtained by subtracting $E(U(\gf{G}))$ from $E(G)$.
\end{proof}

The following lemma shows that it is possible to determine whether there is an edge between the root of a graphlet of size $(n-1)$ and the missing vertex.
\begin{lemma}\label{l:root-missing-edge}
    Let $G$ be a graph, $D$ be the $(\leq n-1)$-gdd of $G$ and $\gf{G} = \gfroot{G \setminus \{x\}}{r}$ be a fixed $(n-1)$-sized graphlet for some $x \in V(G)$. 
    Then we can reconstruct whether $rx \in E(G)$.
    Especially, the number of edges between the automorphism orbit $o$ containing $r$ and $x$ is reconstructible from $D$.
\end{lemma}

\begin{proof}
    Let $\gf{G} = \gfroot{G \setminus \{x\}}{r}$.
    Recall that $U(\gfo{0}) = K_2$ and so $D_{0,r} = \deg_G(\gfo{0},r) = \deg_G(r)$.
    If $\deg_G(r)$ is the same as the degree of the root in $\gf{G}$, then $xr \notin E(G)$. Otherwise, the latter is one less than $\deg_G(r)$ and $xr \in E(G)$.
    Suppose $\vartheta(\gf{G}) = k$.
    All vertices $w$ lying in the same orbit of $G \setminus \{x\}$ have a nonzero value in $D_{w,k}$. For each $w$, we find out whether $wx \in E(G)$. The number of edges detected equals to the number of edges between $x$ and $o$.
\end{proof}

Given a graphlet $\gfroot{H}{r}$, we say that a vertex $v$ lies in  {\it level} $i$ from the root $r$, if $\dist_H(v,r) = i$. All vertices at level $i$ are denoted as $L_i(\gfroot{H}{r})$.

\begin{obs}\label{o:distances}
    Let $v$ a vertex $\in V(G)$ and its graphlet degree sequence be given. Then for each graphlet $\gfo{i} = \gfroot{G \setminus \{v\}}{r}$ of size $(n-1)$ we can infer
    \begin{itemize}
        \item  the distance between $r$ and $v$ in $G$.
        \item  the number of shortest paths between $r$ and  $v$.
        \item the structure of the subgraph of $G$ induced by $\{\bigcup_{i \in [\dist(r,v)-1]} L_i(G)\}$, i.e., the graph induced by all vertices that are closer to $r$ than $v$.
    \end{itemize}
\end{obs}

\begin{proof}
 Let $\dist(r,v) = k$ and let $\mathbf{P_j}$ be the graphlet weakly isomorphic to a path of length $j$ that is rooted at its endpoint.
    Then for $j \in [k-1]$, 
    $$ \graphcount{U(\gfo{i})}{\mathbf{P_j}}{r} = \graphcount{G}{\mathbf{P_j}}{r}.$$
    But then, those counts do not match for $k$, because the paths of length $k$ between $r$ and $v$ is present in $G$, but not in $U(\gfo{i})$. So $\dist(r,v)$ is the smallest $k$ such that the equality above does not hold.
    
    The difference between those two numbers then gives the number of shortest paths between $r$ and $v$.
    The precise structure of the subgraph of $G$ around $r$ up to distance $k-1$ is given by taking the structure of $\gfo{i}$ up to distance $k-1$ from its root $r$.
\end{proof}

In particular, if $v$ is not a universal vertex (which will be dealt with in Theorem \ref{t:univ-vertex}), then  the structure of its neighbourhood, $L_1(H)$, is reconstructible.

\begin{obs}\label{o:subgraphs-G}
Let $G$ a graph and $D$ its $(\leq n-1)$-graphlet degree distribution. Then for each connected induced proper subgraph $S$ of $G$, we can count the number of induced occurrences of $S$ in $G$ from $D$.
\end{obs}
\begin{proof}
Fix $S$, an induced proper subgraph of $G$. The number of occurrences of $S$ in $G$ is obtained by summing all the occurrences of graphlets from $\Graphlet{[S]}$ in $D$ and dividing this number by $|V(S)|$, because each occurrence of $S$ in $G$ 
generates $|V(S)|$ graphlets from $\Graphlet{[S]}$ - one for each rooting of $S$.  
\end{proof}

The following result is basically a graphlet analogue of the result of Lemma 1.1 from \cite{Bondy1969OnUC}. 
\begin{obs}\label{o:subgraphs}
Let $G$ a graph and $D$ its $(\leq n-1)$-graphlet degree distribution. Fix a graphlet $\mathbf{G}$ of size $n-1$ occurring in $D$ such that $U(\gf{G}) = G \setminus \{x\}$ for a fixed $x \in V(G)$.
Then it is possible, for any subgraph $S$ of $G$, to enumerate the number of occurences of $S$ in $G$ such that $x \in S$.
\end{obs}

\begin{proof}
We count the occurrences of $S$ in $G$ using Observation \ref{o:subgraphs-G}.
For graphlet $\gf{G}$, the occurrences of $S$ are counted directly in the graph $U(\gf{G})$.
Clearly, the number of subgraphs of $G$ isomorphic to $S$ containing $x$ is obtained by subtracting
the numbers obtained above.
\end{proof}

Suppose $H$ is a $2$-connected graph, so its deck consists of connected graphs of size $n-1$. Then we can, using Observation \ref{o:subgraphs-G} obtain the deck of $H$.

\begin{lemma}\label{l:deck-from-gdd}
    Given a 2-vertex connected graph $H$ of size $n$ and its $(n-1)$-gdd $D$, we can determine its reconstruction deck of $H$ as defined by Kelly (1957). 
\end{lemma}

\begin{proof}
    Let $\mathcal{H}_{n-1}$ be the set of all induced subgraphs of $H$ of size $n-1$. Note that $\mathcal{H}_{n-1}$ is reconstructible from $D$ because for every $F \in \mathcal{H}_{n-1}$, 
    $\sum_{\mathbf{F} \in \Graphlet{[F]}} \sum_{v \in V(H)} D_{v,\vartheta(\mathbf{F})} > 0$.

    Once $\mathcal{H}_{n-1}$ is obtained, it suffices to count multiplicity of each $F \in \mathcal{H}_{n-1}$ using Observation \ref{o:subgraphs-G}.
\end{proof}

The $2$-vertex connectivity is needed, because in order to detect the occurrence in $H$, we need that every subgraph of size $n-1$ corresponds to a graphlet of size $n-1$, which happens only if the respective subgraph is connected. 
Otherwise, only the connected component of the root is contained in the gdd of $H$ and information about the structure of the remaining parts is lost.

Due to this result, any class of $2$-connected graphs that is reconstructible from the deck is also reconstructible from the $(\leq n-1)$-gdd.

The opposite direction, whether we are able to determine $(n-1)$-gds from the reconstruction deck, is significantly more difficult. Clearly, if the conjecture holds, then we are able to construct the original graph $H$ and determine its $(n-1)$-gds. What is the exact relation between reconstruction deck and graphlet degree distribution, if the conjecture does not hold, is not clear. 

\subsection{Graphlet reconstruction of graphs containing vertex of high or low degree}

Graphs containing a {\it universal vertex}, that is, a vertex connected to all vertices of the given graph, are easily reconstructible.
\begin{thm}\label{t:univ-vertex}
    Let $H$ be a graph containing a universal vertex $v$. Then $H$ can be reconstructed from its $(\leq n-1)$-gdd $D$.
\end{thm}
\begin{proof}
    From $D_{*,0}$, we identify the row corresponding to the universal vertex $v$. Then take any graphlet $\gfo{k}$ of size $n-1$ rooted at any vertex $u \neq v$, such that
    $|E(U(\gfo{k}))| = |E(H)|-n+1$, $|E(H)|$ being reconstructible from $D$ by Observation \ref{o:edges}.
    The reconstruction is done by adding a universal vertex to $U(\gfo{k})$.
\end{proof}

It is also possible to reconstruct graphs with almost-universal vertex.
\begin{thm}\label{t:almost-univ-vertex}
    Let $H$ be a graph on $n\geq 3$ vertices containing a vertex $v$ of degree $n-2$. Then $H$ can be reconstructed from its $(\leq n-1)$-gdd $D$.
\end{thm}
\begin{proof}
    Let $u$ be the only non-neighbour of $v$, i.e. the unique vertex such that $uv \notin E(H)$.
    The key step is to find the graphlet of size $(n-1)$ with the underlying graph $H \setminus \{v\}$ rooted at $u$ for any vertex $v$ with $\deg_H(v) = n-2$.

    Let us reconstruct $|E(H)| = e$ using Observation \ref{o:edges}. Then find the set of $(n-1)$-sized graphlets $\mathcal{H}_{n-1}$ occurring in $H$. Using $\vartheta$ to see their structure, define $\mathcal{H}_{n-1}^{e-n+2} \subseteq \mathcal{H}_{n-1}$ the subset of them having $|E(H)|-n+2$ edges. 
    These are the graphlets with the underlying graphs $H \setminus \{v\}$ for all vertices $v$ with $\deg_H(v) = n-2$. Fix any of them and denote it by $\gfo{k}$.

    The last step is to find a weakly isomorphic graphlet $\gfo{\ell} = \gfroot{H \setminus \{v\}}{r}$ such that $\deg_H(r) = \deg_{H \setminus \{v\}}(r)$, because then $r = u$ is the unique vertex such that $rv \notin E(H)$.
    The reconstruction is finished by adding a new vertex $v$ connected to all vertices of $\gfo{\ell}$ but its root.
\end{proof}

On the other hand, it is also easy to reconstruct graphs with some low degree vertices.
A {\it pendant vertex} (or a {\it leaf}) is a vertex adjacent to exactly one other vertex, i.e. its degree is one.

\begin{thm}\label{t:pendant-vertex}
    Let $H$ be a graph containing a pendant vertex $v$. Then $H$ can be reconstructed from its $(\leq n-1)$-gdd.
\end{thm}
\begin{proof}
    The key step is to find a graphlet of size $(n-1)$ with the underlying graph $G \setminus \{v\}$ rooted at $q$ for any pendant vertex $v$, such that $qv$ is an edge.
    
    At first, we reconstruct the number of edges $e$ of $H$ using Observation \ref{o:edges}.
    Let us denote $\mathcal{H}_{n-1}$ the set of graphlets of size $(n-1)$ that occur in the nonzero columns of the gdd of $H$.
    From the numbering $\vartheta$, the structure and rooting of each $\mathbf{H} \in \mathcal{H}_{n-1}$ is known. Using this information, we build the set $\mathcal{H}_{n-1}^{e-1} \subseteq \mathcal{H}_{n-1}$ of the $(n-1)$-sized graphlets with the underlying graphs having $|E(H)|-1$ edges. These are the graphlets with the underlying graphs $H \setminus \{v\}$ for all pendant vertices $v$.
    For each graphlet $\gfo{i} \in \mathcal{H}_{n-1}^{e-1}$, we compare the degree of the root inside $\gfo{i}$ and its degree in $H$, again by checking $D_{v,0}$ for each root $v$, i.e., we check $D_{v,0}$ in the row containing nonzero value $D_{v,i}$ for the current graphlet $\gfo{i}$.
    Once we find graphlet, where the degree of the root in $H$ is larger than in the graphlet, we have found a graphlet with underlying graph $H \setminus \{v\}$, where $\deg_H(v)=1$, rooted at vertex $q$ adjacent to $v$ in $H$. The reconstruction is finished by adding a pendant vertex adjacent to the root $q$.
\end{proof}

\subsection{Reconstructing separable graphs}
A graph is {\it separable} if it contains a cut vertex.
A {\it block} is a maximal subgraph that is not separable.
It is already known that separable graphs without end vertices, i.e., vertices of degree one, are reconstructible from their decks.\cite{Bondy1969OnUC} 
However, the deck can be easily obtained from the graphlet degree distribution only in the case of $2$-connected graphs, as has been shown in Lemma \ref{l:deck-from-gdd}.
In this section we will show that separable graphs are also reconstructible from their $(\leq n-1)$-graphlet degree distribution.
Note that Theorem \ref{t:pendant-vertex} already shows that any graph containing a vertex of degree one can be reconstructed from its $(\leq n-1)$-gdd.

\begin{thm}\label{t:separable-graphs}
    Let $G$ be a separable graph of order $n \geq 3$. Then $G$ is reconstructible from its $(\leq n-1)$-graphlet degree distribution.
\end{thm}

\begin{proof}
    If $G$ contains a vertex of degree one, it is reconstructible from its $(\leq n-1)$-gdd using Theorem \ref{t:pendant-vertex}.
    Now $G$ has minimal degree $\delta(G) \geq 2$.

    We proceed in the following steps:
    \begin{enumerate}
        \item Detect the index of row corresponding to a cut vertex $c$ in $G$. As has been already noted in the proof of Corollary \ref{c:articulations}, if there are $k$ cut vertices, then they can be identified by having $(n-k)$ graphlets of size $(n-1)$ in their graphlet degree sequence (as opposed to other vertices that have $(n-k-1)$ graphlets of size $(n-1)$ in their graphlet degree sequence).
        \item Identify the structure of $B_1, \dots, B_\ell,\ |B_1| \leq |B_2| \leq \dots \leq |B_\ell|$, the blocks of $G$.
        By having $\delta(G) \geq 2$, it must be the case $|B_i| \geq 3\ \forall i \in [\ell]$.
        Let $\mathcal{G}_{n-1}^c$ be the set of graphlets rooted at $c$ obtained from $G$ by removing a vertex $x \in V(G) \setminus \{c\}$.
        For each $\gf{G} \in \mathcal{G}_{n-1}^c$ there are blocks $B'_1, \dots ,B'_\ell, \ |B'_1| \leq |B'_2| \leq \dots \leq |B'_\ell|$ such that there exists a mapping $\varepsilon: [\ell] \rightarrow [\ell]$ for which 
        $B'_{\varepsilon(j)} \cong B_j$ for all $j \in [\ell], j \neq i$ and $B'_{\varepsilon(i)} \cong B_i \setminus \{x\}$ for some $x \in V(B_i)$.
        Note that this mapping is necessary, as the change in the size of the block $B_i$ after the removal of a vertex may change the ordering of the blocks.
        We say that blocks $B'_j$ with $j\neq i$ are {\it whole}, whereas block $B'_i$ is {\it lacking}.

        We pick a graphlet $\gf{G} \in \mathcal{G}_{n-1}^c$ containg a block $B'_1$ with the smallest size $s$.
        Evidently, this block must be lacking, so $B'_2, \dots, B'_\ell$ are whole. Thus, we have the structure of almost every block of $G$ apart from the smallest one.
        Without loss of generality, $\{B'_2, \dots, B'_\ell\} = \{B_2, \dots, B_\ell\}$ and $B'_1 \subset B_1$.

        Suppose $U(\gf{G}) \cong G \setminus \{x\}$.
        We reconstruct $\deg(x)$ using Corollary \ref{c:degree}.
        To find the structure of $B'_1 \cup \{x\} \cong B_1$, we take the set $\mathcal{B}_1$ containing every possible graph in the form $B'_1 \cup \{y\}$ with $\deg(x)$ edges between $y$ and $B'_1$.
        For each such graph $S \in \mathcal{B}_1$, we compare the number of occurrences of $S$ in $G$ and in $U(\gf{G})$.
        As $\{B'_2, \dots, B'_\ell\} = \{B_2, \dots, B_\ell\}$, the occurrences in blocks different from $B_1$ and $B'_1$ cancel each other out and we obtain the difference between the number of occurrences of $S$ in $B'_1$ and $B_1$.
        Note that $B'_1$ does not contain any $S \in \mathcal{B}_1$,
        because $|B'_1| < |S|\ \forall S \in \mathcal{B}_1$.
        There must be an $S \in \mathcal{B}_1$ such that $S \cong B_1$, because $B_1$ is in the form $B'_1 \cup \{x\}$ and $\mathcal{B}_1$ contains all possible connections between $x$ and $B'_1$.
        If there were $S_1,S_2 \in \mathcal{B}_1$ such that $S_1 \ncong S_2$ and $B_1 \cong S_1$, $B_1 \cong S_2$, then $S_1 \cong B_1 \cong S_2$, a contradiction.
        This means that there exists a unique $S \in \mathcal{B}_1$ that has a different number of occurrences in $G$ and $U(\gf{G})$ and $S \cong B_1$, finishing the reconstruction of $G$.    
    \end{enumerate}
    \end{proof}

\subsection{Graphlet reconstruction of graphs with vertex-deleted asymmetric or almost-asymmetric subgraphs}

When using graphlet degree distribution instead of a deck to reconstruct the original graph, it is possible to utilize the fact that graphlets are rooted. In particular, the rooting allows us to enumerate the number of edges connecting the root and different automorphism orbits (see Lemma \ref{l:root-missing-edge}). In an {\it asymmetric} (or rigid) graph, all automorphism orbits have size one and so this information is sufficient to reconstruct the structure of the graph.
We define an {\it almost-asymmetric graph} to be any graph containing exactly one automorphism orbit of size larger than one.  

\begin{thm}\label{t:asym}
Let $H$ be a $2$-vertex connected graph containing a vertex $v \in V(H)$ for which $H \setminus \{v\} = \Hasym$ is asymmetric or almost-asymmetric such that the following condition holds:

\begin{itemize}
    \item[$(*)$] If there are vertices $v^1, \dots, v^k$, $k \geq 2$, such that $\gmvh{v^i} \cong \gmvh{v} \cong \Hasym\ \forall i \in [k]$,  then every $w \in V(H)$ lies in the same orbit of $\Hasym$ regardless of which $\{v^i\ |\ i \in [k], v^i \neq w\}$ have been removed.
\end{itemize}

Then $H$ is reconstructible from its $(n-1)$-graphlet degree distribution.
\end{thm}

\begin{obs}\label{o:neigh}
    Assume the situation from Theorem \ref{t:asym}. 
    If
    \begin{itemize}
        \item[\rm{(\ding{64})}] $N_H(v^i) = N_H(v^j)\ \forall\, i \neq j; i,j \in [k]$ such that $\Hasym \cong H \setminus \{v^i\}$
    \end{itemize}
then condition $(*)$ is satisfied.
\end{obs}

The condition ({\rm \ding{64}}) can also be formulated in a different way.
Two vertices $u,v \in V(H)$ having the same neighbourhood, i.e. $N_H(u) = N_H(v)$, are called {\it twins}.
Vertices $v, v'$ (or possibly more) for which there exists an automorphism of $H$ mapping $v$ to $v'$, are called {\it similar}. 
If $v$ and $v'$ are not similar but $H \setminus \{v\} \cong H \setminus \{v'\}$, then they are {\it pseudosimilar} \cite{harary-palmer1966pseudosim}.
There exist constructions generating graphs containing such vertices. \cite{Godsil1983pseudosim, Lauri2003pseudosim, Kimble1981pseudosim, Wagner2014pseudosim}
\begin{cor}
    Let $H$ be a $2$-vertex connected graph containing a vertex $v \in V(H)$ such that $H \setminus \{v\}$ is asymmetric or almost-asymmetric
    and one of the following holds
    \begin{enumerate}
        \item There is no vertex $v' \in V(H)$ either similar or pseudosimilar to $v$.
        \item All vertices $v^1, \dots, v^k$, $k \geq 2$ which are either similar or pseudosimilar to $v$ form twins with $v$.
    \end{enumerate} 
    Then $G$ is reconstructible from its $(n-1)$-graphlet degree distribution.
\end{cor}

Before proving Theorem \ref{t:asym}, it is useful to make several observations.

\begin{obs}\label{o:unique-nonzero-value}
Assume the situation from Theorem \ref{t:asym}. 
If condition $(*)$ is satisfied, then $\{ D_{w,p}\ |\ p \in \vartheta(\Graphlet{[\Hasym]})\}$ contains exactly one nonzero value for every $w \in V(H)$.
\end{obs}

\begin{proof}
   Every element of $\{ D_{w,p}\ |\ p \in \vartheta(\Graphlet{[\Hasym]})\}$ corresponds to the number of times $w$ touches a graphlet created by a rooting of $\Hasym$, i.e., the number of orbits of $\Hasym$ in which $w$ occurs.
\end{proof}

\begin{obs}\label{o:D-verifyiable}
All the assumptions of Theorem \ref{t:asym} can be verified using the $(\leq n-1)$-gdd $D$ of $H$.
\end{obs}

\begin{proof}
    Let $\mathcal{D}$ be the set of indices of columns of $D$, that contain a nonzero entry.
    The existence of a vertex $v \in V(H)$ whose removal leads to an asymmetric or almost-asymmetric graph can be verified by simply examining whether $\{U(\gfo{i})\, |\,i \in \mathcal{D}, |V(U(\gfo{i}))| = n-1\}$ contains an assymetric or almost-asymmetric graph. Recall that the ordering $\vartheta$ of graphlets is fixed and known, so the structure of every $\gfo{i}$ is also known.

    Moreover, if $\{D_{w,p}\ |\ p \in \vartheta(\Graphlet{[\Hasym]})\}$ contains a unique nonzero value, then $w$ always touches $\Hasym$ at $\gfo{p}$ corresponding to this value, regardless of which $v^i, i \in k$ have been removed from $H$ to create $\Hasym$ and condition $(*)$ of Theorem \ref{t:asym} is satisfied.
\end{proof}

\begin{proof}[Proof of Theorem \ref{t:asym}]
The reconstruction is done by finding the structure of $\Hasym$, adding a vertex $v$ and then determining for every $w \in V(H) \setminus \{v\} = V(\Hasym)$ whether $wv \in V(H)$.  By using the $2$-connectivity of $H$ and Lemma \ref{l:2-conn-whole-gds-from-n-1}, we obtain $(\leq n-1)$-gdd of $H$ from the $(n-1)$-gdd of $H$. 

Let $\mathcal{H_{\text{n-1}}}$ be the set of graphlets $\gfo{i}$ of size $n-1$ such that $D_{*,i}$ is a nonzero column of $(\leq n-1)$-gdd $D$ of $H$. 
Denote $\mathcal{H}_{\text{asym}} \subset \mathcal{H_{\text{n-1}}}$ the subset of all graphlets  of size $n-1$ occurring in $H$ having asymmetric or almost-asymmetric underlying graphs.
By assumption, there is a nonempty set $\mathcal{H_{\text{cond}}} \subseteq \mathcal{H}_{\text{asym}}$ of graphlets satisfying condition $(*)$. 
Observation \ref{o:D-verifyiable} shows that it is verifiable using $D$ whether $\gfo{i} \in \mathcal{H_{\text{cond}}}$. We fix a graph $\Hasym = U(\gfo{i}), \gfo{i} \in \mathcal{H_{\text{cond}}}$ such that $i = \min\{k\ |\  \gfo{k}\in \mathcal{H_{\text{cond}}}\}$.
Recall that the ordering $\vartheta$ of graphlets is known, which gives us the structure of $\Hasym$.

The next step is to fix a vertex $v \in V(H)$ in $V(H) \setminus V(\Hasym)$. As the orbit at which $w \in V(H)$ touches $\Hasym$ is independent on the choice of $v \in \{v^i\,|\,i \in [k]\}$, we only need to identify any vertex $v^i$ for any $i \in [k]$.

By Observation \ref{o:unique-nonzero-value}, for every $w \in V(H)$, there exists a unique index $p \in \vartheta(\Graphlet{[\Hasym]})$ such that $D_{w,p}$ is nonzero. Then $\gfo{p}$ corresponds to $\Hasym$ rooted at $w$.
Note that if $w \in V(H) \setminus \{v^i\ |\ i \in [k]\}$, then $D_{w,p} = k$. However, if $w \in \{v^i\ |\ i \in [k]\}$, then $D_{w,p} = k-1$, as $w$ never touches $H \setminus \{w\}$. 
So we fix $v = w$ such that $D_{w,p} = k-1$.

    Finally, we reconstruct $H$ by taking $\Hasym$ and introducing a new vertex $v$.
    It only remains to determine whether $wv \in E(H)$ for each $w \in V(\Hasym)$.
    If a vertex $w \in V(\Hasym)$ belongs to an automorphism orbit of size one, then we reconstruct whether $wv \in H$ according to Lemma \ref{l:root-missing-edge}.

    If $\Hasym$ is almost-asymmetric and $w$ belongs to the only automorphism orbit $o$ of size $\ell > 1$, there is a set of $\ell$ nonzero entries corresponding to vertices $w_1, \dots, w_\ell$ in the same automorphism orbit of $\Hasym$. By Lemma \ref{l:root-missing-edge}, the number of edges between $o$ and $v$ is known.
    
    It is not possible to distinguish between $w_i$'s, but as there is only one orbit $o$ of $\Hasym$ size larger than one, the choice of which vertices of $o$ will be connected to $v$ can be made arbitrarily.
    The reason is that as each vertex $u \in V(\Hasym) \setminus \{w_1, \dots, w_\ell\}$ has its own automorphism orbit, $u$ must have the same relation to each vertex of $o = \{w_1, \dots, w_\ell\}$. Otherwise, $w_1, \dots, w_\ell$ would not be in the same automorphism orbit of $\Hasym$.
\end{proof}
As a remark, if we proceeded the same in the case of multiple automorphism orbits of $\Hasym$, the arbitrary choice of edges between vertices of each orbit would be done more than once. However, those choices are not independent, as the choice for first orbit may change the other orbits and thus the vertices would not be indistinguishable anymore.

\section{Uniqness of graphlet degree distributions and sequences}\label{s:uniq}
Graphlet degree distribution can be viewed as a way to fingerprint the given graph. A key question of this viewpoint is: How well does graphlet degree distribution represent the given graph? Clearly, $2-gdd$, i.e. degree sequence, is not a good representation because there are many graphs with the same degree distribution. What about $(\leq n-1)-gdd$ of a graph, is it always unique?
If so, which size $k$ of graphlets is needed to guarantee the uniqueness of $\leq k-gdd$ of the graph in hand? 
By the result of Nýdl \cite{nydl1981}, there exist graphs on $3j+9$ vertices which have the same subgraphs up to size $2j$, so $k$ must be at least $\frac{2n}{3} - 6$.

Slightly harder question revolves around graphlet degree sequences of vertices.
Given two vertices inside the same graph, does equal $(\leq n-1)$-gds imply that those vertices belong to the same orbit?

And what about vertices with the same $(\leq n-1)$-gds in distinct graphs $G$ and $H$. Clearly, $|G| = |H|$, but is it also the case that $G \approx H$ and those vertices are isomorphic images?

The answer to the last question is negative. 
There is a construction leading to two vertices with the same $(\leq n-1)-gds$, which belong to distinct graphs on $n$ vertices, see Figure \ref{f:same-gds-diff-graphs}.
Note that by computer search, this construction is the only example of vertices with the same gds inside distinct graphs for graph sizes at most six.
It is natural to ask whether this is the only way to obtain vertices with the same $(\leq n-1)-gds$ belonging to (not necessarily) distinct graphs.

\begin{figure}
    \centering
    \includegraphics[width=0.4\linewidth]{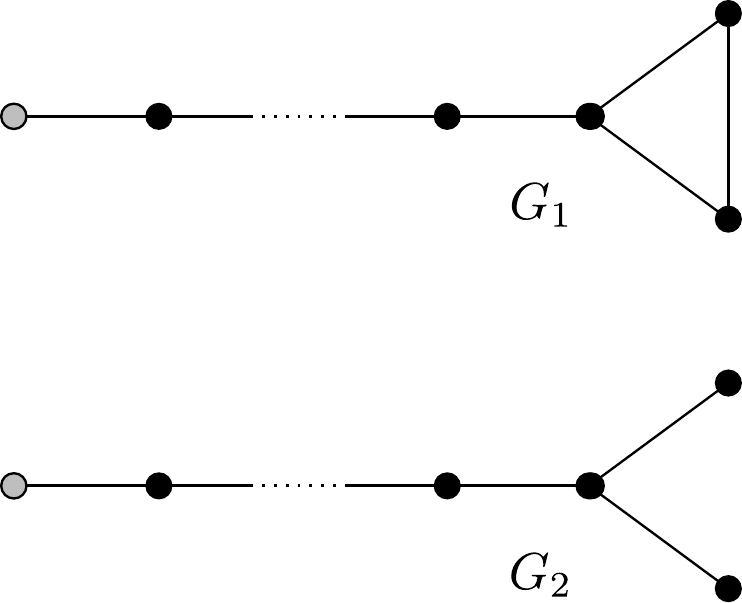}
    \caption{The two graphs $G_1$ and $G_2$ are constracted by taking a path on $n-2$ vertices and ending it with a  triangle or a fork respectively. 
    The grey vertices then have the same gds containing graphlets isomorphic to a path rooted at its end for lengths $1, \dots, n-1$, all such graphlets being represented with multiplicity $1$ with the exception of length $n-1$, which occurs twice. }
    \label{f:same-gds-diff-graphs}
\end{figure}

\section{Deciding whether there exists a graph having given gdd}\label{s:dec}
As the name suggests, graphlet degree sequence generalizes the notion of degree sequence. It is known that $deg_G(v) \in \{ 1, \dots, (n - 1) \}$, where $n = |V(G)|$. 
Not only there is a condition on the values which are obtainable as degrees, there is also a condition related to the whole degree sequence.
Specifically, 
$$ \sum_{v \in V(G)} deg_G(v) \cong 0 \mod 2.$$
The Havel-Hakimi theorem \cite{havel, hakimi} leads to an algorithm which either constructs a graph $G$ with given degree sequence $s$ or truly claims that no such graph exists.

There are dependencies between the graphlet counts of larger and smaller graphlets, as small graphets are included in larger ones, so an occurence of each graphlet (of size larger than 2) implies some number of occurences of smaller graphlets in the graph. 
The dependencies of the graphlet counts for graphlets of sizes 4 and 5 has been studied and applied to improve the time complexity of algorithms enumerating graphlets in a given graph. \cite{orca}

However, here we focus on the relations between the graphlet counts for graphlets of size 2 and 3, because $\{2,3\}$-gds is a natural generalization of a degree sequence. The aim is to obtain assertions that would be helpful in solving the following problem:

\prob{Obtainability of graphlet degree sequence }{matrix $M^{n \times 3}$}{decide whether there exists a graph G s.t. M is $\{2,3\}$-gdd of G}

In applications, it is sometimes useful to be able to generate a graph with prescribed graphlet degree distribution, that is similar to a graphlet degree distribution of a real-world network. Recently, there has been introduced a probabilistic algorithm designed to solve this problem.\cite{gen-graphlet-frequency} 
The problem defined above however aims to recognize that the graphlet degree distribution cannot be realized by any graph, if that is the case.

Let us start with asserting forbidden combinations of values inside graphlet degree of a vertex $v$.
\begin{lemma}
Consider a graph $H$, its vertex $v$ and graphlets $\gf{G^0}, \gf{G^1}$ and $\gf{G^3}$ (according to the function $\gamma$). Then
    $$\deg(v) = \sum_{u \in N_H(v)} \graphcount{H}{\gfo{0}}{u} -  \graphcount{H}{\gfo{1}}{v} - 2\graphcount{H}{\gfo{3}}{v}.$$ 
\end{lemma}
\begin{proof}
    Let $c(u,v)$ be the number of common neighbors of vertices $u$ and $v$.
    If we fix $v \in V(H)$, we can double count twice the number of triangles containing $v$ as follows:
    $$2 \cdot \graphcount{H}{\gfo{3}}{v} = \sum_{u \in N_H(v)} c(u,v).$$
    This holds because $U(\gfo{3})$ is isomorphic to a triangle and because common neighbors $u$ and $v$, $u$ being adjacent to $v$, always form a triangle, where each triangle is counted twice.
    
    As a next step, let us count the graphlet count of $\gfo{1}$ rooted in $v$.
    Note that if there is a path of length three starting in $v$, the only case when it is not induced is when its last vertex is connected to $v$, forming a triangle. From this it follows
    \begin{align*}
        \graphcount{H}{\gf{G_1}}{v} &= \sum_{u \in N_H(v)} \Big(\deg(u) - 1 - c(v,u)\Big)\\
        &= \sum_{u \in N_H(v)} (\graphcount{H}{\gfo{0}}{u} - 1) - \sum_{u \in N_H(v)} c(v, u)\\
        &=  \sum_{u \in N_H(v)} \graphcount{H}{\gfo{0}}{u} - \deg(v) - 2\graphcount{H}{\gfo{3}}{v}\\
        \graphcount{H}{\gfo{0}}{v} = \deg(v) &= \sum_{u \in N_H(v)} \graphcount{H}{\gfo{0}}{u} -  \graphcount{H}{\gfo{1}}{v} - 2\graphcount{H}{\gfo{3}}{v}
    \end{align*}
    
\end{proof}

The occurences of graphlets $\gfo{2}$, a cherry rooted in the middle vertex, and the triangle $\gfo{3}$, are closely related to the structure and size of the neighbourhood of a fixed vertex $v$.

\begin{obs}
    Consider a graph $H$, its vertex $v$  and graphlets $\gf{G^0}, \gfo{2}$ and $\gf{G^3}$ (according to the function $\gamma$). Then
    $$ {\graphcount{H}{\gfo{0}}{v} \choose 2} = \graphcount{H}{\gfo{2}}{v} + \graphcount{H}{\gfo{3}}{v}.$$
\end{obs}
In other words, every edge in $N_H(v)$ corresponds to $v$ touching a triangle and every non-edge in $N_H(v)$ corresponds to $v$ touching a cherry rooted in the middle.
Note that this observation also implies that a vertex $v$ of degree one must have $\graphcount{H}{\gfo{2}}{v} = \graphcount{H}{\gfo{3}}{v} = 0$.

\begin{obs}
    For a graph $H$ and graphlets $\gfo{0}, \gfo{1}$ and $\gfo{3}$ holds the following equality:
    $$\sum_{e=\{u,v\} \in E(H)} c(u,v) = \sum_{v \in V(G)} \graphcount{H}{\gfo{3}}{v},$$
    where $c(u,v)$ denotes the number of common neighbors of $u$ and $v$.
\end{obs}
\begin{proof}
    The equality is based on double counting of the number of triangles.
The left-hand side is counted by first fixing an edge and then counting the triangles incident with it. Note that each triangle is counted three times, once for each its edge.
The right-hand side is by counting the number of triangles touched by each vertex $v$ and realizing that each triangle is counted three times, once for each vertex contained in it.
\end{proof}


\section{Conclusion}
Graphlet degree distributions are a powerful tool providing information about the local structure of the graph. In this work, we examined its properties and showed its potential through its ability to reconstruct the structure of graphs.
Although the set of graphlet-reconstructible graphs is larger than the deck-reconstructible graphs, it remains an open question whether the graplet version of the reconstruction conjecture holds or not.

\section{Acknowledgements}
David Hartman, Aneta Pokorná and Daniel Trlifaj were supported by the Czech Science Foundation Grant No. 23-07074S.

\bibliographystyle{plain}
\bibliography{graphlets}

@article{recnumsur,
    author = {Asciak, K. J. and Francalanza, M.A. and Lauri, J. and Myrvold, W.},
    title = {A Survey of Some Open Questions in Reconstruction Numbers},
    journal = {Ars Combinatoria},
    year = {2010},
    volume = {097},
    pages = {443-456}
}

@article{gimda,
author = {Chen, Xing and Guan, Na-Na and Li, Jian-Qiang and Yan, Gui-Ying},
title = {GIMDA: Graphlet interaction-based MiRNA-disease association prediction},
journal = {Journal of Cellular and Molecular Medicine},
volume = {22},
number = {3},
pages = {1548-1561},
doi = {https://doi.org/10.1111/jcmm.13429},
year = {2018}
}

@article{janssen2012,
author = {Janssen, J. and Hurshman, M. and Kalyaniwalla, N.},
title = {{Model Selection for Social Networks Using Graphlets}},
volume = {8},
journal = {Internet Mathematics},
number = {4},
publisher = {A K Peters, Ltd.},
pages = {338 -- 363},
year = {2012},
}

@article{Finotelli2021,
author = {Finotelli, P. and Piccardi, C. and Miglio, E. and Dulio, P.},
issn = {1662-4548},
journal = {Frontiers in neuroscience},
pages = {665544-},
publisher = {Frontiers Research Foundation},
title = {A Graphlet-Based Topological Characterization of the Resting-State Network in Healthy People},
volume = {15},
year = {2021},
}

@article{erdosrenyi1963asymmetric,
    author = {Erdős, P. and Rényi, A.},
    title = {Asymmetric graphs},
    journal = {Acta Mathematica Academiae Scientiarum Hungaricae},
    volume = {14},
    pages = {295–315},
    year = {1963},
    doi = {https://doi.org/10.1007/BF01895716}
}

@article{orca,
    author = {Hočevar, Tomaž and Demšar, Janez},
    title = {A combinatorial approach to graphlet counting},
    journal = {Bioinformatics},
    volume = {30},
    number = {4},
    pages = {559-565},
    year = {2014},
    month = {12},
    issn = {1367-4803},
    doi = {10.1093/bioinformatics/btt717},
}

@article{nydl1981,
author = {Nýdl, Václav},
journal = {Commentationes Mathematicae Universitatis Carolinae},
language = {eng},
number = {2},
pages = {281-287},
publisher = {Charles University in Prague, Faculty of Mathematics and Physics},
title = {Finite graphs and digraphs which are not reconstructible from their cardinality restricted subgraphs},
url = {http://eudml.org/doc/17107},
volume = {022},
year = {1981},
}

@article{havel,
    author = {Havel, V{á}clav},
    title = {Poznámka o existenci konečných grafů},
    journal = {Časopis pro pěstování matematiky},
    year = {1955},
    pages = {477-480},
    issue = {4},
    issn = {0528-2195},
    volume = {80},
    doi = {10.21136/CPM.1955.108220},
}

@article{hakimi,
author = {Hakimi, S. L.},
title = {On Realizability of a Set of Integers as Degrees of the Vertices of a Linear Graph. I},
journal = {Journal of the Society for Industrial and Applied Mathematics},
volume = {10},
number = {3},
pages = {496-506},
year = {1962},
doi = {10.1137/0110037},
}

@article{przulj2007,
    author = {Pržulj, Nataša},
    title = {Biological network comparison using graphlet degree distribution},
    journal = {Bioinformatics},
    volume = {23},
    number = {2},
    pages = {e177-e183},
    year = {2007},
    month = {01},
    issn = {1367-4803},
    doi = {10.1093/bioinformatics/btl301},
}

@article{orbit-counting,
    doi = {10.1371/journal.pone.0147078},
    author = {Melckenbeeck, Ine AND Audenaert, Pieter AND Michoel, Tom AND Colle, Didier AND Pickavet, Mario},
    journal = {PLOS ONE},
    publisher = {Public Library of Science},
    title = {An Algorithm to Automatically Generate the Combinatorial Orbit Counting Equations},
    year = {2016},
    month = {01},
    volume = {11},
    url = {https://doi.org/10.1371/journal.pone.0147078},
    pages = {1-19},
    number = {1},

}

@article{L-GRAAL,
author = {Malod-Dognin, Noël},
year = {2015},
month = {02},
pages = {},
title = {L-GRAAL: Lagrangian graphlet-based network aligner},
journal = {Bioinformatics},
doi = {10.1093/bioinformatics/btv130}
}

@article{graphettes,
    doi = {10.1371/journal.pone.0181570},
    author = {Hasan, Adib AND Chung, Po-Chien AND Hayes, Wayne},
    journal = {PLOS ONE},
    publisher = {Public Library of Science},
    title = {Graphettes: Constant-time determination of graphlet and orbit identity including (possibly disconnected) graphlets up to size 8},
    year = {2017},
    month = {08},
    volume = {12},
    url = {https://doi.org/10.1371/journal.pone.0181570},
    pages = {1-12},
    number = {8},

}

@article{przulj2004,
    author = {Pržulj, N. and Corneil, D. G. and Jurisica, I.},
    title = {Modeling interactome: scale-free or geometric?},
    journal = {Bioinformatics},
    volume = {20},
    number = {18},
    pages = {3508-3515},
    year = {2004},
    month = {07},
    issn = {1367-4803},
    doi = {10.1093/bioinformatics/bth436},
}

@book{kelly1942isometric,
  title={On Isometric Transformations},
  author={Kelly, P.J.},
  url={http://digital.library.wisc.edu/1793/28652},
  year={1942},
  publisher={University of Wisconsin--Madison}
}

@article{kelly1957congruence,
  title={A congruence theorem for trees.},
  author={Kelly, Paul J.},
  journal={Pacific J. Math.}, 
  pages={961--968},
  year={1957}
}

@article{Giles1974outerplanar,
title = {The reconstruction of outerplanar graphs},
journal = {Journal of Combinatorial Theory, Series B},
volume = {16},
number = {3},
pages = {215-226},
year = {1974},
issn = {0095-8956},
doi = {10.1016/0095-8956(74)90066-5},
author = {Giles, W. B.},
}

@article{Bondy1969OnUC,
  title={On Ulam’s conjecture for separable graphs},
  author={John Adrian Bondy},
  journal={Pacific Journal of Mathematics},
  year={1969},
  volume={31},
  pages={281-288},
}

@article{Lauri1981max-planar,
title = {The reconstruction of maximal planar graphs II. Reconstruction},
journal = {Journal of Combinatorial Theory, Series B},
volume = {30},
number = {2},
pages = {196-214},
year = {1981},
issn = {0095-8956},
doi = {10.1016/0095-8956(81)90064-2},
author = {Lauri, J.},
}

@article{rec-comp-rad-2,
author = {Ramachandran, S. and Monikandan, S.},
year = {2009},
month = {01},
pages = {},
title = {Graph reconstruction conjecture: reductions using complement, connectivity and distance},
volume = {56},
journal = {Bulletin of the Institute of Combinatorics and its Applications}
}

@article{yang1988rec2conn,
author = {Yongzhi, Yang},
title = {The reconstruction conjecture is true if all 2-connected graphs are reconstructible},
journal = {Journal of Graph Theory},
volume = {12},
number = {2},
pages = {237-243},
doi = {https://doi.org/10.1002/jgt.3190120214},
year = {1988}
}

@article{harary-palmer1966pseudosim,
 author = {Harary, F. and Palmer E.},
 journal = {Journal of Mathematics and Mechanics},
 number = {4},
 pages = {623--630},
 publisher = {Indiana University Mathematics Department},
 title = {On Similar Points of a Graph},
 volume = {15},
 year = {1966}
}

@article{Godsil1983pseudosim,
title = {Graphs with three mutually pseudo-similar vertices},
journal = {Journal of Combinatorial Theory, Series B},
volume = {35},
number = {3},
pages = {240-246},
year = {1983},
issn = {0095-8956},
doi = {10.1016/0095-8956(83)90051-5},
author = {C.D Godsil and W.L Kocay},
}

@article{Lauri2003pseudosim,
title = {Constructing graphs with several pseudosimilar vertices or edges},
journal = {Discrete Mathematics},
volume = {267},
number = {1},
pages = {197-211},
year = {2003},
note = {Combinatorics 2000},
issn = {0012-365X},
doi = {10.1016/S0012-365X(02)00615-5},
author = {Josef Lauri},
}

@article{Kimble1981pseudosim,
author = {Kimble Jr., Robert J. and Schwenk, Allen J. and Stockmeyer, Paul K.},
title = {Pseudosimilar vertices in a graph},
journal = {Journal of Graph Theory},
volume = {5},
number = {2},
pages = {171-181},
doi = {10.1002/jgt.3190050207},
year = {1981}
}

@article{Wagner2014pseudosim,
    author = {Wagner, S. and Wang, H.},
    title = {Indistinguishable Trees and Graphs},
    journal = {Graphs and Combinatorics},
    year = {2014},
    volume = {30},
    issue = {6},
    doi = {10.1007/s00373-013-1360-6},
    pages = {1435-5914},
}

@article{protein-align-Przulj,
    author = {Malod-Dognin, Noël and Pržulj, Nataša},
    title = {GR-Align: fast and flexible alignment of protein 3D structures using graphlet degree similarity},
    journal = {Bioinformatics},
    volume = {30},
    number = {9},
    pages = {1259-1265},
    year = {2014},
    month = {01},
    issn = {1367-4803},
    doi = {10.1093/bioinformatics/btu020},
}

@article{AlonMotifs,
    author = {Shen-Orr, S. and  Milo, R. and Mangan, S. and Alon, U.},
    title = {Network motifs in the transcriptional regulation network of Escherichia coli},
    journal = {Nature Genetics},
    volume = {31},
    pages = {64-68},
    year = {2002},
    doi = {10.1038/ng881},
}

@article{gen-graphlet-frequency,
    doi = {10.1371/journal.pone.0328639},
    author = {Mornie, B. AND Colle, D. AND Audenaert, P. AND Pickavet, M.},
    journal = {PLOS ONE},
    publisher = {Public Library of Science},
    title = {Generating random graphs with prescribed graphlet frequency bounds derived from probabilistic networks},
    year = {2025},
    month = {08},
    volume = {20},
    pages = {1-22},

    number = {8},
}

\end{document}